\documentclass[]{interact}
\usepackage{appendix}
\usepackage{xcolor}
\usepackage{fancyhdr}
\usepackage{enumerate}
\usepackage{float}
\usepackage{hyperref}
\usepackage{latexsym}
\usepackage{comment}
\usepackage{anyfontsize} 
\usepackage{tikz} 
\usetikzlibrary{cd} 
\usepackage{color}
\usepackage{hyperref}
\usepackage{url}
\usepackage{mathtools}
\usepackage{breakurl}
\newcommand{\bburl}[1]{\textcolor{blue}{\url{#1}}}
\newcommand{\ve}[1]{\vec{\mathbf{#1}}}
\newcommand{\abs}[1]{\left| #1 \right|}

\newcommand\norm[1]{\left\lVert#1\right\rVert}
\def\<{\langle}
\def\>{\rangle}

\newcommand{\E}{\mathbb{E}}
\newcommand{\Var}{\operatorname{Var}}
\newcommand{\Prob}{\mathbb{P}}

\makeatletter
\newcommand{\monthyear}[1]{%
  \def\@monthyear{\uppercase{#1}}}
\newcommand{\volnumber}[1]{%
  \def\@volnumber{\uppercase{#1}}}
\makeatother

\newcommand{\R}{{\mathbb R}}

\newcommand{\C}{{\mathbb C}}
\newcommand{\N}{{\mathbb N}}
\newcommand{\Z}{{\mathbb Z}}

\theoremstyle{plain}
\numberwithin{equation}{section} 
\newtheorem{thm}{Theorem}[section] 
\newtheorem{theorem}[thm]{Theorem}
\newtheorem{lemma}[thm]{Lemma}
\newtheorem{example}[thm]{Example}
\newtheorem{definition}[thm]{Definition}
\newtheorem{proposition}[thm]{Proposition}
\newtheorem{corollary}[thm]{Corollary}
\newtheorem{conjecture}[thm]{Conjecture}
\newtheorem{remark}[thm]{Remark}

\numberwithin{table}{section} 
\numberwithin{figure}{section}

\begin{document}

\monthyear{Month Year}
\volnumber{Volume, Number}
\setcounter{page}{1}

\title{Properties of Multidimensional Vector Zeckendorf Representations}

\author{
\name{Ivan Bortnovskyi\textsuperscript{a}, Julian Duvivier\textsuperscript{b}, Pedro Espinosa\textsuperscript{c}, Michael Lucas\textsuperscript{a}, Steven J. Miller\textsuperscript{c}, Tiancheng Pan\textsuperscript{a}, Arman Rysmakhanov\textsuperscript{c}, Iana Vranesko\textsuperscript{c}, Ren Watson\textsuperscript{d}\thanks{Ren Watson: renwatson@utexas.edu},  and Steven Zanetti\textsuperscript{e}}
\affil{\textsuperscript{a}Department of Pure Mathematics and Mathematical Statistics, University of Cambridge, Cambridge, United Kingdom, CB3 0WA; \textsuperscript{b}Department of Mathematics, Reed College, Portland, OR, 97202; \textsuperscript{c}Department of Mathematics, Williams College, Williamstown, MA, 01267; \textsuperscript{d}Department of Mathematics, University of Texas at Austin, Austin, TX 78712;
\textsuperscript{e}Department of Mathematics, University of Michigan, Ann Arbor, MI 48109} \thanks{The first listed author was supported by The Winston Churchill Foundation of the United States. The NSF grant DMS2242623 supported the second, third, seventh, eighth, ninth, and tenth listed authors, and the third, seventh, and eighth authors were supported by Williams College and the Finnerty fund. Finally, the fourth and sixth listed authors were supported by the Dr. Herchel Smith Fellowship Fund.}
}

\maketitle

{\bf Article type}: research 
\bigskip

\begin{abstract}
	Zeckendorf's Theorem says that for all $k \geq 3$, every nonnegative integer has a unique \emph{$k$-Zeckendorf representation} as a sum of distinct $k$-bonacci numbers, where no $k$ consecutive $k$-bonacci numbers are present in the representation. Anderson and Bicknell-Johnson \cite{anderson2011multidimensional} extend this result to the multidimensional context: letting the \emph{$k$-bonacci vectors} $\ve{X}_i \in \Z^{k-1}$ be given by $\ve{X}_0=\ve{0}$, $\ve{X}_{-i}=\ve{e}_i$ for $1 \leq i \leq k-1$, and $\ve{X}_n=\sum_{i=1}^k \ve{X}_{n-i}$ for all $n \in \Z$, they show that for all $k \geq 3$, every $\ve{v} \in \mathbb Z^{k-1}$ has a unique \emph{$k$-bonacci vector Zeckendorf representation}, a sum of distinct $k$-bonacci vectors where no $k$ consecutive $k$-bonacci vectors are present in the representation. Their proof provides an inductive algorithm for finding such representations. We present two improved algorithms for finding the $k$-bonacci vector Zeckendorf representation of $\ve{v}$ and analyze their relative efficiency. We utilize a projection map $S_n:\mathbb Z^{k-1} \to \mathbb Z_{\geq 0}$, introduced in \cite{anderson2011multidimensional}, that reduces the study of $k$-bonacci vector representations to the setting of $k$-bonacci number representations, provided a lower bound is established for the most negatively indexed $k$-bonacci vector present in the $k$-bonacci vector Zeckendorf representation of $\ve{v}$. Using this map and a bijection between $\mathbb Z^{k-1}$ and $\mathbb Z_{\geq 0}$, we further show that the number of and gaps between summands in $k$-bonacci vector Zeckendorf representations exhibit the same properties as those in $k$-Zeckendorf representations and that $k$-bonacci vector Zeckendorf representations exhibit summand minimality.
\end{abstract}


\section{Introduction}
\subsection{History and Motivation}
A beautiful theorem of Zeckendorf \cite{originalzeckendorf} states that every nonnegative integer $n$ can be written uniquely as a sum of non-consecutive Fibonacci numbers. \footnote{ We define these by $F_1 =1, F_2 = 2$ and $F_{n+1} = F_n + F_{n-1}$; otherwise, we lose uniqueness. } We refer to this representation as the \textbf{\textit{Zeckendorf decomposition}} of $n$. Zeckendorf decompositions and their generalizations have been the subject of extensive previous study; for instance, see \cite{beckwith2013average, bower2015gaps, cordwell2018summand, kologlu2011summands, li2017collection, millerwang2012CLT, miller2014gaussian}. One natural extension of the Fibonacci sequence is as follows, where the Fibonacci numbers are given by taking $k=2$. 

\begin{definition}[$k$-bonacci Sequence]
    For a fixed choice of $k$, the \textbf{$k$-bonacci sequence} $\{x_n\}$ \footnote{ Because $k$ is fixed, we suppress it in the notation for ease of reading.} is given by $x_n=0$ for $-k+2\leq n \leq 0$, $x_1=1$, and $x_n=\sum_{i=1}^kx_{n-i}$ for all $n \in \Z$.
\end{definition}

The original proof of Zeckendorf's Theorem, via the greedy algorithm for $k=2$, naturally extends to $k \geq 3$, giving the following result. Note that the restriction $i \geq 2$ is imposed to guarantee $x_1=x_2=1$ are not both allowed in the decomposition, as this would result in loss of uniqueness of decomposition.

\begin{theorem}[Zeckendorf]\label{thm:kbonacci-zeckendorf}
    Every nonnegative integer $n$ can be written uniquely as a sum of distinct $k$-bonacci numbers $n=\sum_{i \geq 2}c_i x_i$ such that $c_i\in \{0,1\}$ for all $i$ and no $k$ consecutive $c_i$'s are equal to $1$.
\end{theorem}

For a fixed $k$, we refer to the unique representation of $n$ given by Theorem \ref{thm:kbonacci-zeckendorf} as the \textbf{\textit{$k$-Zeckendorf representation of $n$}}. Formally, the greedy algorithm used to find the $k$-Zeckendorf representation of a positive integer $n$ is as follows.

\begin{definition}[$k$-bonacci Number Greedy Algorithm]
    For a fixed $k > 1$ and positive integer $n$, the \textbf{\textit{$k$-bonacci number greedy algorithm}} finds the unique $k$-Zeckendorf representation of $n$ as follows.
    \begin{enumerate}
        \item Initialize $R\coloneq n$.
        \item Find $\ell_{1} \geq 2$ maximal such that $x_{\ell_{1}} \leq R$.
        \item Reset $R\coloneq R-x_{\ell_1}$.
        \item If $R=0$, the algorithm terminates. Else, repeat step (2) to find $\ell_{2} \geq 2$ maximal such that $x_{\ell_{2}} \leq R$ and reset $R\coloneq R-x_{\ell_2}$. Repeat this process until $R=0$. The finite sum $\sum_{i \geq 1} x_{\ell_i}$ forms the $k$-Zeckendorf representation of $n$.
    \end{enumerate}
\end{definition}

This paper extends the study of multidimensional Zeckendorf representations initiated by Anderson and Bicknell-Johnson \cite{anderson2011multidimensional}. Before stating our results, we introduce relevant notation and prior results.

\begin{definition}[$k$-bonacci Vectors] \label{def: k-bonacci vectors}
    The \textbf{$k$-bonacci vectors} $\ve{X}_i \in \Z^{k-1}$ are given by $\ve{X}_0=\ve{0}$, $\ve{X}_{-i}=\ve{e}_i$ for $1 \leq i \leq k-1$, and $\ve{X}_n=\sum_{i=1}^k \ve{X}_{n-i}$ for all $n \in \Z$.
\end{definition}

The notion of $k$-Zeckendorf representations for nonnegative integers extends naturally to representations of vectors $\ve{v} \in \mathbb Z^{k-1}$ as sums of $k$-bonacci vectors.

\begin{theorem}\cite[Theorem 2]{anderson2011multidimensional}\label{thm:vector-rep}
    Every $\ve{v} \in \Z^{k-1}$ has a unique representation $\ve{v}=\sum_{i \geq 1}c_i\ve{X}_{-i}$ such that $c_i\in \{0,1\}$ for all $i$ and no string of $k$ consecutive $c_i$’s is equal to $1$.
\end{theorem}

For instance, taking $k=3$, the first several 3-bonacci vectors are given by: \begin{align*}
   \ve{X}_{-1}&~=~(1,0),\\\ve{X}_{-2}&~=~(0,1),\\\ve{X}_{-3}&~=~(-1,-1),\\\ve{X}_{-4}&~=~(2,0),\\\ve{X}_{-5}&~=~(-1,2),\\\ve{X}_{-6}&~=~(-2,-3),\\\ve{X}_{-7}&~=~(5,1).
\end{align*}Then the unique representation of $(7,0)$ satisfying the criteria of Theorem \ref{thm:vector-rep} is given by \begin{align*}(7,0)&~=~(5,1)+(2,0)+(-1,-1)+(1,0)\\&~=~\ve{X}_{-7}+\ve{X}_{-4}+\ve{X}_{-3}+\ve{X}_{-1}.\end{align*}

\begin{definition}[Satisfying Representation]
    For $\ve{v} \in \Z^{k-1}$, a representation $\ve{v}=\sum_{i \geq 1}c_i\ve{X}_{-i}$ such that all $c_i \in \{0,1\}$ and no string of $k$ consecutive $c_i$'s are equal to $1$ is called a \textbf{satisfying representation (SR)} of $\ve{v}$, and the unique such representation of $\ve{v}$ may be denoted $SR(\ve{v})$. A representation $\ve{v}=\sum_{i \geq 1}c_i\ve{X}_{-i}$ such that all $c_i \in \{0,1,2\}$, no string of $k$ consecutive $c_i$'s are nonzero, and only one string of consecutive nonzero $c_i$'s contains any $2$'s is called a \textbf{nearly satisfying representation (NSR)} of $\ve{v}$.
\end{definition}

Anderson and Bicknell-Johnson \cite{anderson2011multidimensional} introduced the following scalar product, which will allow us to utilize known results regarding the $k$-bonacci numbers in our analysis of the $k$-bonacci vectors.

\begin{definition}[$k$-bonacci Projection Map]\label{defn: s_n}
    For $n \geq k-2$, let $S_n:\Z^{k-1} \to [0,x_n)$ be the linear map $$S_n(\ve{v})~=~\ve{v} \cdot (x_{n-1},...,x_{n-(k-1)}) \pmod{x_n}.$$
\end{definition}

\begin{lemma}\cite[Lemma 3]{anderson2011multidimensional}\label{lem:projection} We have $S_n(\sum_{i=1}^pc_i\ve{X}_{-i}) \equiv\sum_{i=1}^pc_ix_{n-i} \pmod{x_n}$.
\end{lemma}
\begin{proof}
    For $0\leq i\leq k-1$, $$S_n(\ve{X}_{-i}) ~=~ \ve{e_i}\cdot(x_{n-1},\dots,x_{n-(k-1)}) =  x_{n-i}.$$ For $i\geq k$, we have $S_n(\ve{X}_{-i})\equiv x_{n-i} \pmod{x_n}$ by the recursive definitions of $\ve{X}_{-i}$ and $x_{n-i}$. Linearity completes the proof.
\end{proof}

\begin{lemma}\cite[Corollary 5]{anderson2011multidimensional}\label{lem:termbounding}
  Suppose $\ve{v}=\sum_{i=1}^Mc_i\ve{X}_{-i}$ and $\ve{v}'=\ve{v}+\ve{X}_{-p}=\sum_{i=1}^{M'}c'_i\ve{X}_{-i}$ are SR's and that $p \leq M$. Then $M'\leq k+M$.
\end{lemma}

For a fixed choice of $k \geq 2$ and $\ve{v} \in \mathbb Z^{k-1}$, Anderson and Bicknell-Johnson’s \cite{anderson2011multidimensional} proof of Theorem \ref{thm:vector-rep} provides a recursive approach to finding the SR of $\ve{v}$. A primary focus of this paper is to provide an improved algorithmic approach for finding the SR of $\ve{v}$. Before stating our results, we recall the following definition.

\begin{definition}[Big-$O$ Notation]
    Let $f(x)$ and $g(x)$ be real-valued functions. We say that $f(x)=O(g(x))$ if there exists some constant $c > 0$ and $x_0 \in \mathbb R$ such that $|f(x)| \leq c|g(x)|$ for all $x \geq x_0$.
\end{definition}

\subsection{Main Results}

In order to compare our algorithmic approaches to finding the SR of $\ve{v}$, we first provide in Appendix \ref{appendix:complexity} a more formal definition and complexity analysis of the algorithm given by \cite{anderson2011multidimensional}, yielding the following result.

\begin{lemma}\label{lem:old-algo-complexity}
    Let $L$ be the number of summands in the SR of $\ve{v} \in \mathbb Z^{k-1}$. The recursive algorithm for finding the SR of $\ve{v}$ in \cite{anderson2011multidimensional} runs in $O(k\cdot L\cdot \norm{\ve{v}}_1)$ steps.
\end{lemma}

In Section \ref{sec:findrep}, we establish two improved algorithmic approaches for finding the SR of $\ve{v}$ and provide an analysis of their relative efficiency. Both algorithms rely on the following definition.

\begin{definition}[Maximal $k$-bonacci index]
    The \textbf{\textit{maximal $k$-bonacci index}} of $\ve{v} \in \Z^{k-1} \backslash \{\ve{0}\}$ is
    $$ J(\ve{v}) ~\coloneq~  \max \{ i \in \N\ |\  \ve{X}_{-i}\  \text{appears in the SR of $\ve{v}$} \}. $$
\end{definition}

Both of our algorithms utilize the following modification of the $k$-bonacci number greedy algorithm. 

\begin{definition}[Vector Zeckendorf Greedy Algorithm]
    Let $k>1$ and $\ve{v} \in \Z^{k-1} \backslash \{ \ve{0} \}$. Suppose that $j \geq J(\ve{v})$. Then the \textbf{\textit{vector Zeckendorf greedy algorithm}} for $\ve{v}$ proceeds as follows.
    \begin{enumerate}
        \item Apply $S_{j+1}(\ve{v}) = \ve{v} \cdot (x_j, \dots, x_{j+1-(k-1)}) \pmod{x_{j+1}}$.
        \item Use the $k$-bonacci number greedy algorithm to write $S_{j+1}(\ve{v}) = \sum_{i=1}^j c_i x_{j+1-i} \pmod{x_{j+1}}$, where $c_i \in {0,1}$ and no $k$ consecutive $c_i$’s are equal to $1$.
        \item Output the representation $\sum_{i=1}^j c_i \ve{X}_{-i}$.
    \end{enumerate}
    
\end{definition}

Thus, our algorithmic approaches rely on finding efficient methods for obtaining a value of $j$. The following lemma allows us to compute a value for $j$ from any $k$-bonacci vector decomposition of $\ve{v}$.

\begin{lemma}
    \label{lem:inductive-termbounding}
    Suppose $\ve{v} = \sum_{i=1}^N \ve{X}_{-n_i}$ is any $k$-bonacci decomposition of $\ve{v}$ (not necessarily an SR). Then
    \begin{equation} \label{eq:inductive-termbounding-j}
        k(N-1) + \max_{1\leq i \leq N} n_i ~\geq~ J(\ve{v}),
    \end{equation}
    so we can use $j = k(N-1) + \max_i n_i$ in the vector Zeckendorf greedy algorithm.
\end{lemma}

\begin{proof}
    The proof follows from inductive use of Lemma \ref{lem:termbounding}.
\end{proof}

The different values for $j$ arise from different ways of constructing a $k$-bonacci decomposition of $\ve{v}$. The first relies on taking a series of small steps, consisting of a combination of $\ve{X}_{-k}$ and the standard basis vectors $\{\ve{X}_{i}\}_{i=1}^{k-1}$, and obtains a larger upper bound on $j$ more quickly.

\begin{definition}[Small Steps Bound for $j$]
    Fix $\ve{v}=(v_1, v_2, \dots, v_{k-1}) \in \Z^{k-1} \setminus \{\ve{0}\}$. Proceed according to the following cases.
    \begin{enumerate}
        \item If $\ve{v}$ contains some entry $v_i<0$, then let $v_m$ be the maximally negative entry of $\ve{v}$. First, take $v_m$ steps in direction $\ve{X}_{-k}=(-1,-1,...,-1)$; then, take $v_j-v_m$ steps in direction $\ve{X}_{-j}=\ve{e}_j$ for each $j \in \{1,...,k\}$.
        \item If $\ve{v}$ contains no negative entries, then  take $v_j$ steps in direction $\ve{X}_{-j}$ for each $j \in \{1,...,k\}$.
    \end{enumerate}
\end{definition}
    This algorithm gives rise to the following bound on $j$.
\begin{theorem}
\label{thm:small-step-bound}
    Let $\ve{v}=(v_1, v_2, \dots, v_{k-1}) \in \Z^{k-1} \setminus \{\ve{0}\}$. If $v_i < 0\ \text{for some } i$, then let $v_m$ be the negative entry of $\ve{v}$ that is largest in absolute value. Let $j_{ssb}$ be defined by
    $$ j_{ssb} ~\coloneq~ \left\{
    \begin{array}{rcl}
        |v_m|k+\sum_{i=1}^{k-1}(v_i-v_m)k & \text{if} & v_i < 0\ \text{for some } i\\
        k\sum_{i=1}^{k-1}v_i -1 & \text{if} & v_i \geq 0\ \text{for all } i.
    \end{array}
    \right. $$
    Then $j_{ssb} \geq J(\ve{v})$.
\end{theorem}

We verify in Proposition \ref{prop:complexity-small-steps} that this approach runs in $O(k^3\norm{\ve{v}}_\infty)$ steps. We can make the bound on $j$ given by the left-hand side in \eqref{eq:inductive-termbounding-j} smaller by taking fewer steps. This motivates a second approach, in which each step is chosen to bring us as close as possible to $\ve{v}$. Formally, this is given below.

\begin{definition}[Large Steps Bound for $j$]
    \label{def:michael-algo}
    Fix $\ve{v} \in \Z^{k-1}$. Let $\ve{v}_1 = \ve{v}$, and given $\ve{v}_i$, take $n_i \in \N$ such that $\norm{\ve{v}_i - \ve{X}_{-n_i}}_2$ is minimal, and let $\ve{v}_{i+1} = \ve{v}_i - \ve{X}_{-n_i}$. Stop when $i=M+1$ and $\norm{\ve{v}_{M+1}}_2 \geq \norm{\ve{v}_M}_2$. Then, use the algorithm given in the small steps bound to write $\ve{v}_M = \sum_{i=M}^N \ve{X}_{-n_i}$ with $ n_i \in \{1, \dots, k\}$. This gives a decomposition $\ve{v} = \sum_{i=1}^N \ve{X}_{-n_i}$.\\
    \\
    We use Lemma \ref{lem:inductive-termbounding} to compute
    $$ j_{lsb} ~\coloneq~ k(N-1) + \max_i n_i ~\geq~ J(\ve{v}).$$
\end{definition}

When $k=3$, the large steps bound allows us to obtain a logarithmic bound for $j$.\\
\\
We provide Python code carrying out these algorithms in the following \href{https://colab.research.google.com/drive/1MNQNCxUA83AshL2vjkcXCuvpVs8Kw1kH?usp=sharing}{Colab file:} \\
\url{https://colab.research.google.com/drive/1MNQNCxUA83AshL2vjkcXCuvpVs8Kw1kH?usp=sharing}
.

\begin{theorem}\label{thm:large-steps-j-bound-k=3}
    Fix $k=3$. There exist $c_3, d_3>0$ such that for all $\ve{v} \in \Z^2 \backslash \{ \ve{0} \}$,
    $$ j_{lsb} ~\leq~ c_3 \log \norm{\ve{v}}_2 + d_3.$$
    Hence, we have $j_{lsb} = O(\log \norm{\ve{v}}_2) = O(\log \norm{\ve{v}}_\infty)$ by the Lipschitz equivalence of $L^p$ norms.
\end{theorem}

We conjecture based on strong computational evidence that this result may be extended to all $k \geq 3$.

\begin{conjecture} \label{conj:large-steps-j-bound}
    Fix $k \geq 3$. There exist $c_3, d_3>0$ such that for all $\ve{v} \in \Z^2 \backslash \{ \ve{0} \}$,
    $$ j_{lsb} ~\leq~ c_3 \log \norm{\ve{v}}_2 + d_3.$$
    Hence, we have $j_{lsb} = O(\log \norm{\ve{v}}_2) = O(\log \norm{\ve{v}}_\infty)$ by the Lipschitz equivalence of $L^p$ norms.
\end{conjecture}

In Section \ref{sec:gaps}, we turn our attention to the number of and gaps between summands in $k$-bonacci vector Zeckendorf representations. We show that they exhibit the same properties as those in the one-dimensional case. 
\begin{definition}\label{defn: d_n}
For $n\ge 0$, let $D_n \ \coloneq\  \left\{ \vec{\mathbf{v}}\in\mathbb Z^{k-1}: J(\ve{v})\le n  \right\}$.
\end{definition}
\begin{definition}\label{def:assocprobspace}

For $n>0$, consider the discrete outcome space $D_{n} \setminus D_{n-1}$ with probability measure $$ \mathbb{P}_n(A)  \ = \  \frac{|A|}{x_{n+2}-x_{n+1}}, \quad  A \subset D_{n}\setminus D_{n-1}. $$ Note $|D_{n} \setminus D_{n-1}| = x_{n+2} - x_{n+1}$, so this is the uniform measure. We define the random variable $K_n$ by setting $K_n(\ve{v})$ equal to the number of summands of $\ve{v} \in D_{n} \setminus D_{n-1}$ in its SR.
\end{definition}

\begin{definition}
Let $\ve{v}\in D_{n}\setminus D_{n-1}$, and let $\kappa(\ve{v})$ denote the number of summands in the SR of $\ve{v}$. Let $\{r_j\}_{j=1}^{\kappa(\ve{v})}$ be the set of indices present in the SR of $\ve{v}$.
\begin{itemize}
\item \emph{Spacing gap measure:} We define the spacing gap measure of $\ve{v}$ by  \begin{equation} \nu_{\ve{v};n}(x)  \ := \  \frac{1}{\kappa(\ve{v})-1} \sum_{j=2}^{\kappa(\ve{v})} \delta\left(x - (r_j - r_{j-1})\right), \end{equation} where $\delta$ is the Dirac delta functional.\footnote{Thus for any continuous function $f$ we have $\int_{-\infty}^\infty f(x) \delta(x-a)dx = f(a)$; we may view $\delta(x-a)$ as representing a unit point mass concentrated at $a$.} We do not include the gap to the first summand, as this is not a gap \emph{between} summands; for almost all $\ve{v}$, one extra gap is negligible in the limit.

\item \emph{Average spacing gap measure:} The SR of $\ve{v}$ has $\kappa(\ve{v})-1$ gaps. Thus the total number of gaps for all $\ve{v}\in D_{n}\setminus D_{n-1}$ is 
\begin{equation} N_{{\rm gaps}}(n) \ := \ \sum_{\ve{v}\in D_{n}\setminus D_{n-1}} \left(\kappa(\ve{v}) - 1\right).\end{equation}
We define the average spacing gap measure for all $\ve{v} \in D_{n}\setminus D_{n-1}$ by \begin{align} \nu_n(x) & \ :=\ \frac1{N_{{\rm gaps}}(n)} \sum_{\ve{v}\in D_{n}\setminus D_{n-1}} \sum_{j=2}^{\kappa(\ve{v})} \delta\left(x - (r_j - r_{j-1})\right) \nonumber\\ &\  = \ \frac1{N_{{\rm gaps}}(n)} \sum_{\ve{v}\in D_{n}\setminus D_{n-1}} \left(\kappa(\ve{v})-1\right) \nu_{\ve{v};n}(x).\end{align} If $P_n(l)$ is the probability of getting a gap of length $l$ among all gaps from the decompositions of all $\ve{v}\in D_{n}\setminus D_{n-1}$, then \begin{equation} \nu_n(x) \ = \ \sum_{l=0}^{n-1} P_n(l) \delta(x - l). \end{equation}

\end{itemize}
\end{definition}

\begin{theorem}\label{thm: number_of_summands}
    Let $K_n$ be the random variable of Definition \ref{def:assocprobspace} and denote its mean by $\mu_n$. Then there exist constants $C_{\rm{Lek}}>0$ (the constant in the Generalized Lekkerkerker's Theorem for PLRS's \cite[Theorem 1.2 (Generalized Lekkerkerker)]{miller2014gaussian}), $d$, and $\gamma_1\in (0,1)$ depending only on $k$ such that
    $$\mu_n \ = \ C_{\rm{Lek}}n+d+o(\gamma_1^n). $$
    The mean $\mu_n$ and variance $\sigma_n^2$ of $K_n$ grow linearly in $n$, and $(K_n-\mu_n)/\sigma_n$ converges weakly to the standard normal $N(0,1)$ as $n\rightarrow \infty$.
\end{theorem}

The following is a standard lemma which helps us state a result on the distribution of gaps.

\begin{lemma}[Generalized Binet's Formula]\label{lem:binet}
Let $\lambda_1,\dots,\lambda_k$ be the roots of the characteristic polynomial for the $k$-bonacci sequence, which is
\begin{equation}\label{eq:char_poly}
p(x) \ \coloneq\  x^k - x^{k-1} - \dots - x - 1 = 0.
\end{equation}
Order the roots so that $|\lambda_1| \ge \dots \ge |\lambda_k|$. Then $\lambda_1>1$ is the unique positive real root, and there exist constants such that
\begin{equation}\label{eq:generalized_binet}
x_{n+1} \ =\ a_1 \lambda_1^n + O(n^{k-2} \lambda_2^n).
\end{equation}
\end{lemma}

\begin{theorem}\label{thm:avegapsbulk} Let $\lambda_1 > 1$ denote the largest root (in absolute value) of the characteristic polynomial for the $k$-bonacci sequence, and let $a_1$ be the leading coefficient in the generalized Binet expansion. Let $P_n(l)$ be the probability of having a gap of length $l$ among the decompositions of $\ve{v} \in D_n \setminus D_{n-1}$, and let $P(l) = \lim_{n\to\infty} P_n(l)$. Then
\begin{equation}
P(l) \ = \  \begin{cases}
0 & \text{if } l = 0 \\
\lambda_1^{-1}(\frac{1}{C_{\rm Lek}})(\lambda_1(1-2a_1) + a_1) & \text{if } l=1 \\
(\lambda_1-1)^2 \left(\frac{a_1}{C_{\rm Lek}}\right) \lambda_1^{-l} & \text{if } l \ge 2.
\end{cases}
\end{equation} 
In particular, the probability of having a gap of length $l \ge 2$ decays geometrically, with decay constant the largest root of the characteristic polynomial.
\end{theorem}

We further extend the study of summand minimality to the multidimensional case, giving the following.

\begin{theorem} \label{thm: vec_sum_minimality}
    Let $\ve{v} \in \Z^{k-1}$. Then $SR(\ve{v})=\sum_{i \geq 1}c_i\ve{X}_{-i}$ is summand minimal: that is, there is no way to write $\ve{v}$ as a linear combination of $k$-bonacci vectors with nonnegative integer coefficients using strictly fewer terms.
\end{theorem}


\section{Improved Algorithms for Finding Vector Representations}\label{sec:findrep}

In this section, we fix a choice of $k>1$ and $\ve{v} \in \Z^{k-1} \backslash \{ \ve{0} \}$, noting that the SR of $\ve{0}$ is the empty representation.
We provide two algorithms to compute the SR of $\ve{v}$ which improve the algorithmic approach of \cite{anderson2011multidimensional}. We first show how to use the greedy algorithm to find the vector decomposition of $\ve{v}$ for a specific $j$.

\begin{lemma}
\label{lem:vec-zeck-greedy-algo}
    Let $k>1$ and $\ve{v} \in \Z^{k-1} \backslash \{ \ve{0} \}$. Let $S_{j+1}(\ve{v}) = \sum_{i=1}^j c_i x_{j+1-i} \pmod{x_{j+1}}$ be the output of the vector Zeckendorf greedy algorithm for $\ve{v}$. Then $\sum_{i=1}^j c_i \ve{X}_{-i}$ is the SR of $\ve{v}$. 
\end{lemma}

\begin{proof}
    This follows because the map $S_{j+1}$ is injective on the set of vectors with SRs whose indices are bounded by $j$.
\end{proof}

\subsection{Small Steps Algorithm}

Let $k>1$ and $\ve{v} \in \Z^{k-1} \backslash \{ \ve{0} \}$. Lemma \ref{lem:vec-zeck-greedy-algo} requires us to have a value of $j$ before using the vector Zeckendorf greedy algorithm to generate the SR of $\ve{v}$. We verify the small steps bound on $j$ by writing $\ve{v}$ as a linear combination of $\{ \ve{X}_{-i} \}_{i=1}^k$ and applying Lemma \ref{lem:inductive-termbounding}.

\begin{proof}[Proof of Theorem \ref{thm:small-step-bound}]
    First, suppose that $v_i <0$ for some $i \in \{1,...,k-1\}$. Note that for any choice of $k > 1$, $\ve{X}_{-k}=(-1,-1,...,-1)$. Hence we can write
    $$ \ve{v} ~=~ \abs{v_m} \ve{X}_{-k} + \sum_{i=1}^{k-1} (v_i-v_m) \ve{X}_{-i}. $$
    By Lemma \ref{lem:inductive-termbounding}, we have 
    $$J(\ve{v}) ~\leq~ |v_m|k+\sum_{i=1}^{k-1}(v_i-v_m)k \ = \ j_{ssb},$$
    Otherwise, $v_i \geq 0$ for every $i \in \{1,...,k-1\}$. In this case, we write
    $$ \ve{v} ~=~ \sum_{i=1}^{k-1} v_i \ve{X}_{-i}. $$
    Again, by Lemma \ref{lem:inductive-termbounding},
    $$ J(\ve{v}) ~\leq~ k\left(\sum_{i=1}^{k-1} v_i - 1 \right) + k - 1
    ~=~ k\sum_{i=1}^{k-1}v_i -1 ~=~ j_{ssb}.$$
\end{proof}

\begin{remark}
    The bound $j_{ssb}$ generated by the small steps algorithm is larger than necessary even for small $k$ and vectors with relatively small integer entries. The following examples illustrate this.
\end{remark}

\begin{example}
    \label{ex:(2,-2)}
    Consider $\ve{v}=(2,-2)$. As $k=3$, we get $j=2(3)+4(3)=18$. Then $$S_{19}(2,-2) ~\equiv~ 2(19513)-2(10609) \pmod{35890} ~\equiv~ 17808 \pmod{35890}.$$ The $k$-bonacci number greedy algorithm gives the decomposition \begin{align*}17808&~=~10609+5768+927+504\\&~=~x_{17}+x_{16}+x_{13}+x_{9}\\&~=~x_{19-2}+x_{19-3}+x_{19-6}+x_{19-7}.\end{align*} Then the SR of $(2,-2)$ is $\ve{X}_{-2}+\ve{X}_{-3}+\ve{X}_{-6}+\ve{X}_{-7}=(0,1)+(-1,-1)+(-2,-3)+(5,1)$.
\end{example}
\begin{example}
    Consider $\ve{v}=(3,0)$. As $k=3$, we get $j=3\cdot3-1=8$. Then $S_{9}(3,0)=3\cdot44 \pmod{81} = 51$. The $k$-bonacci number greedy algorithm gives the decomposition  \begin{align*}51&~=~44+7\\&~=~x_{8}+x_{5}\\&~=~x_{9-1}+x_{9-4}.\end{align*} So the SR of $(3,0)$ is $\ve{X}_{-1}+\ve{X}_{-4}=(1,0)+(2,0)$.
\end{example}
The following proposition is proved in Appendix \ref{appendix:complexity}.
\begin{proposition}
\label{prop:complexity-small-steps}
    The time complexity of the small steps algorithm is $O(k^3 \norm{\ve{v}}_\infty)$.
\end{proposition}
The time complexity of the algorithm in \cite{anderson2011multidimensional} is $O(k\cdot L\cdot \norm{\ve{v}}_1)$ (see Appendix \ref{appendix:complexity}), where $L$ is the number of summands in the decomposition of $\ve{v}$. For fixed $k$, this is $O(\norm{\ve{v}}_\infty \log\norm{\ve{v}}_\infty)$. As well as being slow, calculating the SR of $\ve{v}$ requires us to calculate all of the SRs along a path from $\ve{0}$ to $\ve{v}$.

The time complexity of the Small Steps Algorithm is $O(k^3 \norm{\ve{v}}_\infty)$ [see Appendix \ref{appendix:complexity}]. For fixed $k$, it is $O(\norm{\ve{v}}_\infty)$. Although this is an improvement over the algorithm in \cite{anderson2011multidimensional}, Example \ref{ex:(2,-2)} demonstrates that we end up working with very large numbers, even for small $k$ and $\ve{v}$ with small integer entries. This arises from the linear bound for $J(\ve{v})$. This motivates our second algorithm, which finds an improved bound on $j$.

\subsection{Large Steps Algorithm}

We first illustrate the Large Steps Algorithm with an example.

\begin{example}
    Consider $\ve{v}=(2,-2)$. We let $\ve{v}_1 = (2,-2)$. The closest 3-bonacci vector to $\ve{v}$ is $\ve{X}_{-4}=(2,0)$. Then we take $\ve{v}_2 = \ve{v}_1 - \ve{X}_{-4} = (0,-2)$. The closest 3-bonacci vector to $\ve{v}_2$ is $\ve{X}_{-3}=(-1,-1)$, so we set $\ve{v}_3 = \ve{v}_2 - \ve{X}_{-3} = (1,-1)$. The closest 3-bonacci vector to $\ve{v}_3$ is $\ve{x}_{-1} = (1,0)$, so we set $\ve{v}_4 = \ve{v}_3 - \ve{X}_{-1} = (0,-1)$. Now, subtracting off the closest 3-bonacci vector to $\ve{v}_4$ would not reduce the size, so we use the small steps algorithm to write $\ve{v}_4 = \ve{X}_{-3} + \ve{X}_{-2}$. This gives the decomposition
    $$ \ve{v} ~=~ \ve{X}_{-4} + 2\ve{X}_{-3} + 2\ve{X}_{-1},$$
    and Lemma \ref{lem:inductive-termbounding} tells us that
    $$ j_{lsb} ~=~ 4 + 3 \cdot 4 = 16. $$
    We now proceed as before with the vector Zeckendorf greedy algorithm.
\end{example}

\begin{example}
    Consider $\ve{v} = (3,0)$. We let $\ve{v}_1 = (3,0)$. The closest 3-bonacci vector to $\ve{v}$ is $\ve{X}_{-4} = (2,0)$, so we set $\ve{v}_2 = \ve{v}_1 - \ve{X}_{-4} = (1,0) = \ve{X}_{-1}$. This gives the decomposition
    $$ \ve{v} ~=~ \ve{X}_{-4} + \ve{X}_{-1}, $$
    which translates to
    $$ j_{lsb} ~=~ 4 + 3 \cdot 1 = 7$$
    by Lemma \ref{lem:inductive-termbounding}. We then proceed with the vector Zeckendorf greedy algorithm.
\end{example}

\subsection{Analysis of Large Steps Algorithm}

We now utilize a geometric argument to prove that the value $j_{lsb}$ provided by the large step algorithm is logarithmic in $\norm{\ve{v}}_\infty$ for $k=3$. We conjecture that this is true for all $k$.

\begin{proof}[Proof of Theorem \ref{thm:large-steps-j-bound-k=3} ]
    We prove the following statements.
    \begin{enumerate}
        \item There exist natural numbers $P$ and $M$ such that for any sequence of $P$ consecutive 3-bonacci vectors starting at $\ve{X}_{-n}$ for $n\geq M$, and any wedge in $\R^2$ centered about $\ve{0}$ with angle $2\pi/3$, one of the 3-bonacci vectors lies in the wedge. \label{statement:spiraliness}
        \item There exist $\lambda<1$ and $R>0$ such that for any $\ve{v} \in \Z^2$ with $\norm{\ve{v}}_2 \geq R$, there exists $\ve{X}_{-n}$ with
        $$ \norm{\ve{v}-\ve{X}_{-n}}_2 ~\leq~ \lambda \ve{v}.$$ \label{statement:close-fib}
        \item There exist constants $c_3$ and $d_3$ such that
        $$j_{lsb} ~\leq~ c_3 \log \norm{\ve{v}}_2 + d_3.$$ \label{statement:log-bound}
    \end{enumerate}

    Proof of \eqref{statement:spiraliness}: We claim that $P=M=9$ are sufficient for $k=3$. We first rewrite the 3-bonacci vectors in $\C$, letting $Z_{-n} = \Phi(X_{-n})$, where $\Phi(v_1,v_2) = v_1 + i v_2$. Then $Z_{-n}$ satisfies the recurrence relation $Z_{-n}=Z_{-(n-3)}-Z_{-(n-2)}-Z_{-(n-1)}$ for $n>2$ with $Z_0=0$, $Z_{-1}=1$, and $Z_{-2}=i$. Then, we solve the recurrence relation to get
    $$ Z_{-n} ~=~ A r^n e^{in\theta} + B r^n e^{-in\theta} + C \epsilon^n,$$
    where \begin{align*}
        r &~=~ 1.3562, \\
        \theta &~=~ 2.1762, \\
        \epsilon &~=~ 0.5437, \\ A &~=~ -0.4578-0.3103i,\\ B &~=~ -0.0612-0.0259i,\\ C &~=~ 0.5190+0.3362i,
    \end{align*} each to 4 decimal places.\\
    \\
    \begin{figure}[H]
    \centering
    \includegraphics[width=0.6\linewidth]{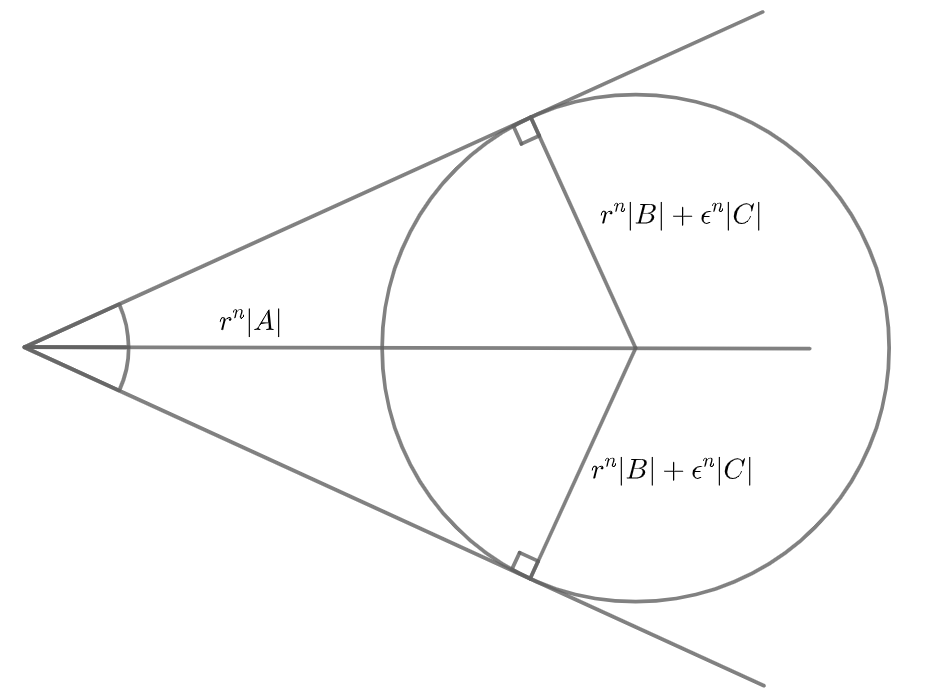}
    \caption{A figure displaying a circle in which $Z_{-n}$ must fall. We use this to approximate the argument of $Z_{-n}$.}
    \label{fig:angle-bound}
    \end{figure}
    Figure \ref{fig:angle-bound} shows that the difference in angle between $Z_{-n}$ and $Ae^{in\theta}$ is at most $\arcsin \left(\frac{\abs{B} + \abs{C}(\epsilon/r)^n}{\abs{A}}\right)$, which decreases monotonically in $n$ to
    $$ \arcsin \frac{\abs{B}}{\abs{A}} ~\approx~ 0.1204 \quad \text{(to 4 decimal places).}$$ 
    For $n \geq 9$, we have
    $$ \arcsin \left(\frac{\abs{B} + \abs{C}(\epsilon/r)^n}{\abs{A}} \right) ~\leq~ 0.121.$$
    Hence, for $n,m \geq 9$, we have
    \begin{equation} \label{eq:spiraliness-bound}
    \arg(Z_{-m}) - \arg(Z_{-n}) \in ((m-n)\theta -0.242, (m-n)\theta + 0.242).
    \end{equation}
    Consider for $n \geq 9$, the numbers $Z_{-n}, Z_{-(n+1)}, \dots, Z_{-(n+8)}$. To show that every cone with semiangle $\pi/3$ contains one of these numbers, we show that the angle between every two adjacent lines in Figure \ref{fig:nine-vecs} is less than $2\pi/3$ .
    
    \begin{figure}[H]
    \centering
    \includegraphics[width=0.5\linewidth]{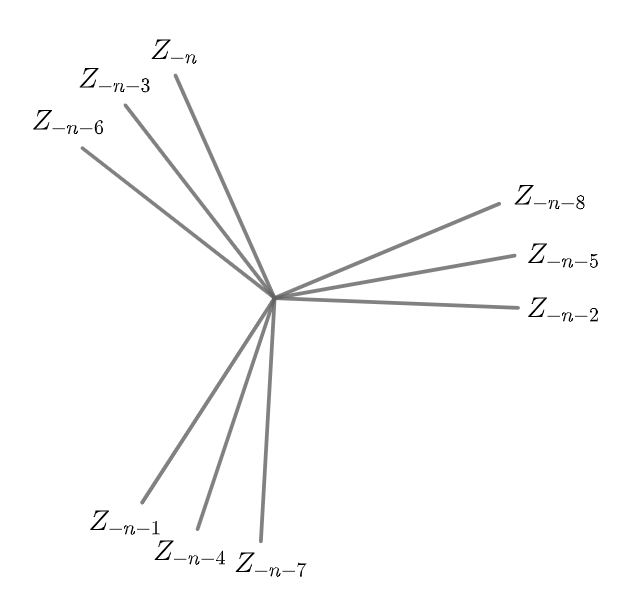}
    \caption{Nine consecutive $Z_{-i}$'s. We aim to show that the angle between each pair of adjacent lines is less than $2\pi/3.$}
    \label{fig:nine-vecs}
    \end{figure}
    By \eqref{eq:spiraliness-bound}, we have
    \begin{align*}
    \arg(Z_{-(n+6)}) - \arg(Z_{-n}) &\in (0.2490, 0.7331) \subseteq (0,2\pi/3) \\
    \arg(Z_{-(n+1)}) - \arg(Z_{-(n+6)}) &\in (1.4432, 1.9273) \subseteq (0,2\pi/3) \\
    \arg(Z_{-(n+7)}) - \arg(Z_{-(n+1)}) &\in (0.2490, 0.7331) \subseteq (0,2\pi/3) \\
    \arg(Z_{-(n+2)}) - \arg(Z_{-(n+7)}) &\in (1.4432, 1.9273) \subseteq (0,2\pi/3) \\
    \arg(Z_{-(n+8)}) - \arg(Z_{-(n+2)}) &\in (0.2490, 0.7331) \subseteq (0,2\pi/3) \\
    \arg(Z_{-n}) - \arg(Z_{-(n+8)}) &\in (1.1976, 1.6817) \subseteq (0,2\pi/3).
    \end{align*}
    This completes the proof of \eqref{statement:spiraliness}.\\
    \\
    Proof of \eqref{statement:close-fib}: By the triangle inequality, we have
    $$ |A|r^n - |B|r^n - |C| \epsilon^n ~\leq~ \|\ve{X}_{-n}\|_2 ~\leq~ |A|r^n + |B|r^n + |C|\epsilon^n. $$
    This shows that $\|\ve{X}_{-n}\|_2$ is increasing for $n>1$.\\
    \\
    For $\ve{v}$ with $\|\ve{v}\|_2 \geq 2(|A|r^{P+M-1} + |B|r^{P+M-1} + |C| \epsilon^{P+M-1}) = 682496$, let $X_{-n}$ be the smallest 3-bonacci vector such that the upper bound of its norm is greater than $\frac{1}{2}\norm{\ve{v}}_2$. Consider some 3-bonacci vector $\ve{X}_{-n'}$ (with $N \leq n-P \leq n' \leq n-1$) that lies in the cone with axis passing through $\ve{v}$ and semiangle $\pi/3$. We have
    \begin{align*}
    \|\ve{X}_{-n'}\|_2 &~\geq~ |A|r^{n'} - |B|r^{n'} - |C| \epsilon^{n'} \\
    &~\geq~ \frac{|A|r^{n'} - |B|r^{n'} - |C| \epsilon^{n'}}{|A|r^n + |B|r^n + |C| \epsilon^n} \frac{1}{2}\norm{\ve{v}}_2 \\
    &~\geq~ \frac{|A|r^{n-P} - |B|r^{n-P} - |C| \epsilon^{n-P}}{|A|r^n + |B|r^n + |C| \epsilon^n} \frac{1}{2}\norm{\ve{v}}_2 \\
    &~\geq~ \frac{|A|r^{M} - |B|r^{M} - |C| \epsilon^{M}}{|A|r^{M+P} + |B|r^{M+P} + |C| \epsilon^{M+P}}\frac{1}{2} \norm{\ve{v}}_2,
    \end{align*}
    since the fraction in the penultimate line is increasing in $n$ and $n \geq M+P$. Plugging in our values of $A,B,C,r,\epsilon$ and rounding down, we have
    $$ \|\ve{X}_{-n'}\|_2 ~\geq~ \frac{1}{2} \gamma \norm{\ve{v}}_2, $$
    where $\gamma = 7.714 \times 10^{-4}$. This shows that the region $E$ in Figure \ref{fig:guaranteed-region} contains a 3-bonacci vector. We therefore bound the distance from $\ve{v}$ to $\ve{X}_{-n'}$ by the distance from $\ve{v}$ to the furthest point from $\ve{v}$ in $E$. Hence, we have that
    \begin{align*}
    \|\ve{v}-\ve{X}_{-n'}\|_2 &~\leq~ \sqrt{\left(\frac{\sqrt{3}}{2}\norm{\ve{v}}_2\right)^2 + \left( \frac{1}{2} (1-\gamma) \norm{\ve{v}}_2^2 \right)} \\
    &~\leq~ \frac{1}{2}\sqrt{3+(1-\gamma)^2} \norm{\ve{v}}_2 \\
    &~=~ \lambda \norm{\ve{v}}_2,
    \end{align*}
    where $\lambda = 0.9998073$. 

    \begin{figure}[H]
    \centering
    \includegraphics[width=0.5\linewidth]{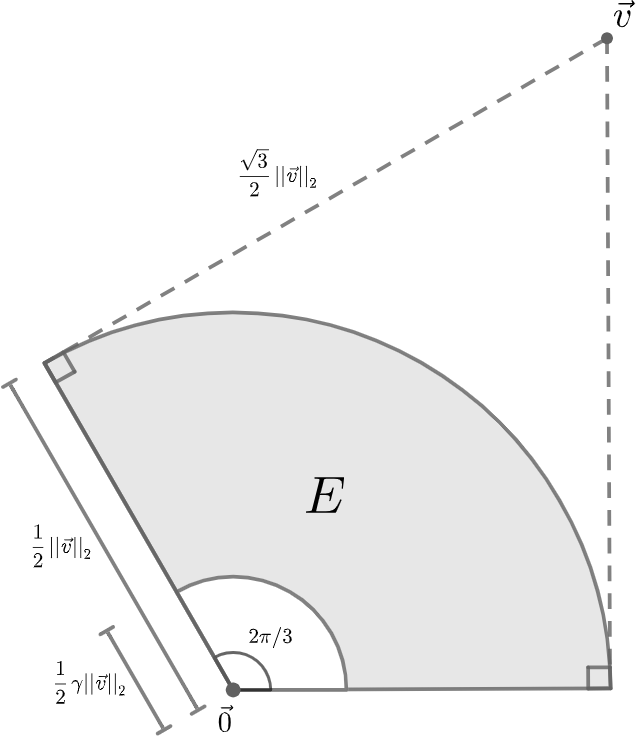}
    \caption{Figure showing a region $E$ which is guaranteed to contain a 3-bonacci vector.}
    \label{fig:guaranteed-region}
    \end{figure}
    
    Proof of \eqref{statement:log-bound}: The large step algorithm gives a decomposition
    \begin{equation}
    \ve{v} ~=~ \sum_{i=1}^{N} \ve{X}_{-n_i}.
    \end{equation}
    To prove that $N$ is finite, from the previous statement we obtain
    \begin{equation}\label{eq:algoineqs}
    \norm{\ve{v}_i}_2 ~\leq~ \lambda \norm{\ve{v}_{i-1}}_2 ~\leq~ \cdots ~\leq~ \lambda^{i-1} \norm{\ve{v}}_2.
    \end{equation}
    Now, for $k=3$, at most the last two terms are from the case where no largest step exists, so $\ve{v}_{N-1}$ has norm at least 1, and none of the $\ve{v}_i$'s for $i<N-1$ are $(-1,0)$ or $(0,-1)$. Therefore, $1 \leq \lambda^{N-2} \norm{\ve{v}}_2$, so we obtain the bound
    \begin{equation}\label{eq:Nbound}
    N ~\leq~ \frac{\log(\norm{\ve{v}}_2)}{\log( 1/\lambda)} + 2.
    \end{equation}
    By Lemma \ref{lem:inductive-termbounding}, we have
    \begin{equation}
    j_{lsb} ~\leq~ 3(N-1) + \max_{i \in \{0,1, \dots, N+1\} } n_i.
    \end{equation}
    Finally, we bound the possible values of $n_i$ that can appear in the decomposition of $\ve{v}$. Statement \eqref{statement:close-fib} ensures that $\|\ve{X}_{-n_i}\|_2 \leq 2 \norm{\ve{v}}_2$, and we know by the triangle inequality that $\|\ve{X}_{-{n_i}}\|_2 \geq (|A|-|B|)r^{n_i} - |C| \epsilon^{n_i} \geq (|A|-|B|)r^{n_i} - |C| \epsilon$. Hence, $(|A|-|B|)r^{n_i} \leq 2 \norm{\ve{v}}_2 + |C| \epsilon$, so
    \begin{equation}
    n_i ~\leq~ \frac{\log \left(\frac{2 \norm{\ve{v}}_2+|C|\epsilon}{|A|-|B|}\right)}{\log r}.
    \end{equation}
    Putting this all together, we have
    \begin{align*}
    j_{lsb} &~\leq~ 3\left(\frac{\log(\norm{\ve{v}}_2)}{\log( 1/\lambda)} + 1 \right) + \frac{\log \left(\frac{2 \norm{\ve{v}}_2+|C|\epsilon}{|A|-|B|}\right)}{\log r} \\
    &~\leq~ 3\left(\frac{\log(\norm{\ve{v}}_2)}{\log( 1/\lambda)} + 1 \right) + \frac{\log \left(\frac{2 +|C|\epsilon}{|A|-|B|}\norm{\ve{v}}_2\right)}{\log r} \\
    &~=~ \left( \frac{3}{\log(1/\lambda)} + \frac{1}{\log r}\right) \log(\norm{\ve{v}}_2) + \left( 3 + \frac{\log(2+|C|\epsilon) - \log(|A|-|B|)}{\log r} \right).
    \end{align*}
    Plugging values in for $\lambda, r, A, B,$ and $C$, we have $c_3 \approx 15570$ and $d_3 \approx 5.018$, rounded up to 4 significant figures.
\end{proof}

This proof not only shows the existence of constants $c_3$ and $d_3$, but also gives us their values. However, the value of $c_3$ exhibited in this proof is too large for practical use. A computation over all vectors $\ve{v}$ with $\norm{\ve{v}}_\infty \leq 100$ suggests that $c_3 = 15,d=10$ is sufficient. We also support this with a scatter plot for 1000 randomly generated vectors with $L^\infty$ norm at most 10000. \\
\begin{figure}[H]
    \centering
    \includegraphics[width=0.7\linewidth]{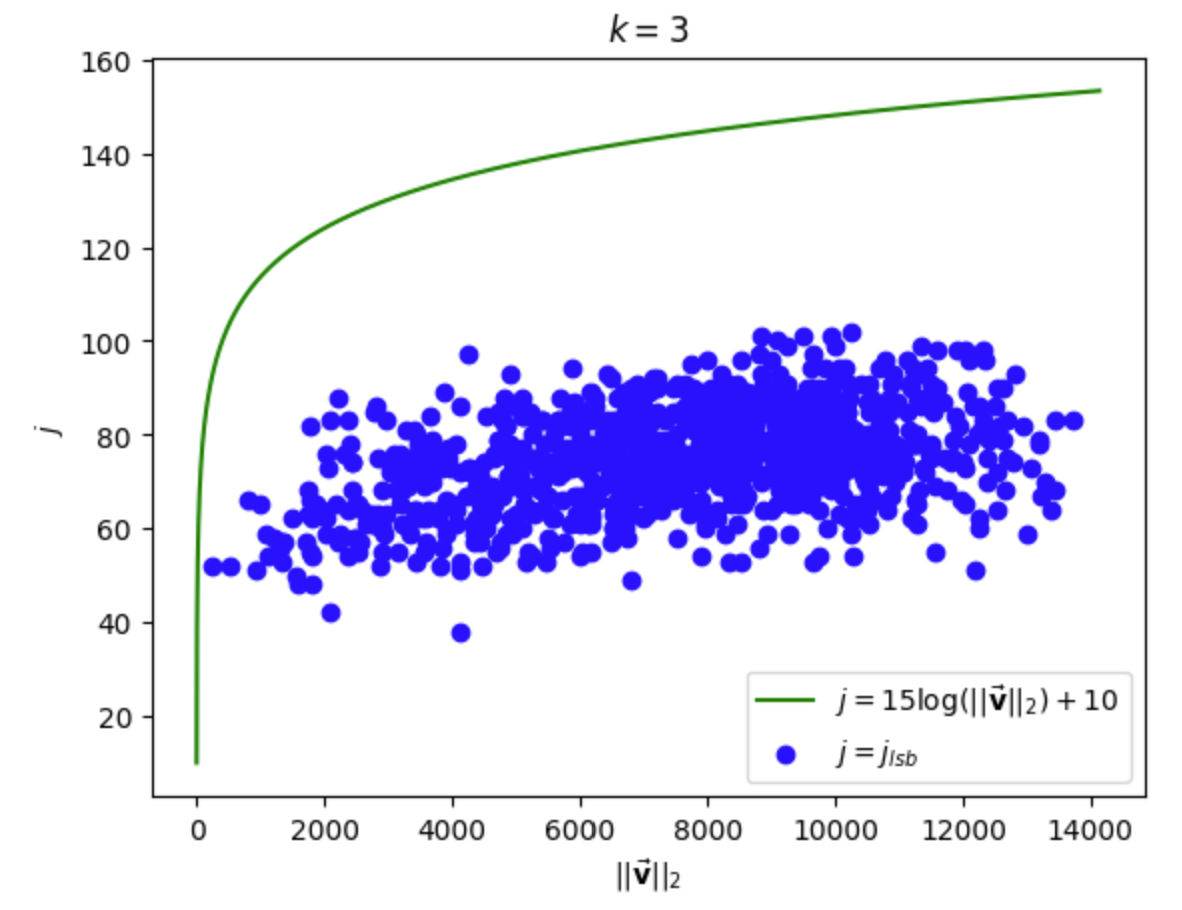}
    \caption{Scatter plot of $j_{lsb}$ and the proposed upper bound with 1000 randomly generated vectors}.
    \label{fig:scatter-j-bound}
\end{figure}

We give the time-complexity of the algorithm in the next proposition.

\begin{proposition} \label{prop:complexity-large-steps}
    For fixed $k$, the time complexity of the large step algorithm is $O(\norm{\ve{v}}_\infty)$.
\end{proposition}

We prove Proposition \ref{prop:complexity-large-steps} in Appendix \ref{appendix:complexity}.

\subsection{Relative Efficiency of Algorithms for Finding SRs}

In summary, we have two methods to find a value of $j$ for the vector Zeckendorf greedy algorithm. For each fixed $k$, $j_{ssb} = O(\norm{\ve{v}}_\infty)$, but assuming Conjecture \ref{conj:large-steps-j-bound}, we have $j_{lsb}=O(\log\norm{\ve{v}}_\infty)$. Considering $k$ fixed, both methods have the same time-complexity, but the slowest part of each method is the greedy algorithm. We gain more insight into the running time by considering how long it takes to calculate $j_{ssb}$ and $j_{lsb}$. The proofs in Appendix \ref{appendix:complexity} show that the calculation of $j_{ssb}$ takes $O(1)$ operations, and that the calculation of $j_{lsb}$ takes $O((\log \norm{\ve{v}}_\infty)^2)$ operations for fixed $k$.


\section{Properties of Vector Representations}\label{sec:gaps}

\subsection{Overview}

The one-dimensional $k$-Zeckendorf representation possesses many interesting properties. The number of summands in decompositions of integers in $[x_n,x_{n+1})$ converges to a Gaussian distribution, and the distribution of gaps between summands follows geometric decay. Furthermore, the decompositions exhibit summand minimality; that is, there is no way to express any nonnegative integer $n$ as a linear combination of $k$-bonacci numbers with nonnegative integer coefficients using strictly fewer terms than the $k$-Zeckendorf representation. 

We show that vector Zeckendorf representations have extremely similar properties (Theorems \ref{thm: number_of_summands}, \ref{thm:avegapsbulk}, and \ref{thm: vec_sum_minimality}). A useful strategy is to reduce the problem to the one-dimensional case by considering appropriate bijections between subsets of $\Z^{k-1}$ and subsets of $\Z_{\ge 0}$. An example of such a map is $S_n$ in Definition \ref{defn: s_n}. The following map, first observed by Anderson and Bicknell-Johnson \cite{anderson2011multidimensional}, also provides valuable insight.

\begin{remark}\label{rem:vector-integer-bijection}
    The satisfying sequences $\{c_i \}$ in Theorem \ref{thm:vector-rep} are essentially the same as $k$-Zeckendorf representations for non-negative integers (as in Theorem \ref{thm:kbonacci-zeckendorf}). This gives a bijection between $\Z^{k-1}$ and $\Z_{\ge 0}$.
\end{remark}

Recall that $D_n =  \left\{ \vec{\mathbf{v}}\in\mathbb Z^{k-1}: J(\ve{v})\le n  \right\}$, as given in Definition \ref{defn: d_n}. Then, the following result is a more precise statement of Remark \ref{rem:vector-integer-bijection}.

\begin{lemma}\label{Lemma: simple_bijection}
    Define $f:\Z^{k-1} \to \Z_{\ge 0}$ as follows. For any $\ve{v}\in \Z^{k-1}$, let $\ve{v} = \sum_{i\ge 1} a_i \ve{X}_{-i}$ be its (unique) SR. Let
    \begin{equation}\label{eq:simple_bijection}
    f(\ve{v}) \ \coloneq\  \sum_{i\ge 2} a_{i-1} x_{i}.
    \end{equation}
    Then $f$ is a bijection. For $n>0$, the image of $D_n$ under $f$ is $[0,x_{n+2})\cap \Z$.
\end{lemma}
\begin{proof}
    Bijectivity follows from Theorem \ref{thm:kbonacci-zeckendorf} and Theorem \ref{thm:vector-rep}. For the final statement, note that the non-negative integers with a $k$-Zeckendorf decomposition of the form $\sum_{i= 2}^{n+1} c_i x_i$ are exactly the ones in $[0,x_{n+2})$.
\end{proof}

The $k$-bonacci numbers belong to the class of \textbf{\textit{Positive Linear Recurrence Requences}}, which have been studied extensively (see \cite{beckwith2013average,bower2015gaps,kologlu2011summands,millerwang2012CLT}). For these sequences, the properties of decompositions are well-understood.

\begin{definition}\label{defn:PLRSdef} \cite{bower2015gaps} We say a sequence $\{H_n\}_{n=1}^\infty$ of positive integers is a \textbf{Positive Linear Recurrence Sequence (PLRS)} if the following properties hold.

\begin{enumerate}

\item \emph{Recurrence relation:} There are non-negative integers $L, c_1, \dots, c_L$ such that $$H_{n+1} \ = \ c_1 H_n + \cdots + c_L H_{n+1-L},$$ with $L, c_1$ and $c_L$ positive.

\item \emph{Initial conditions:} $H_1 = 1$, and for $1 \le n < L$ we have
$$H_{n+1} \ =\
c_1 H_n + c_2 H_{n-1} + \cdots + c_n H_{1}+1.$$

\end{enumerate}

We call a decomposition $\sum_{i=1}^{m} {a_i H_{m+1-i}}$ of a positive integer $N$ (and the sequence $\{a_i\}_{i=1}^{m}$) \textbf{legal} if $a_1>0$, the other $a_i \ge 0$, and one of the following two conditions holds:

\begin{itemize}

\item We have $m<L$ and $a_i=c_i$ for $1\le i\le m$.

\item There exists $s\in\{0,\dots, L\}$ such that
\begin{equation}\label{scon}
a_1\ =\ c_1,\ \ \ a_2\ =\ c_2,\ \ \ \cdots,\ \ \ a_{s-1}\ = \ c_{s-1}\ \ \ {\rm{and}}\  \ \ a_s\ < \ c_s,
\end{equation} $a_{s+1}, \dots, a_{s+\ell} = 0$ for some $\ell \ge 0$,
and $\{b_i\}_{i=1}^{m-s-\ell}$ (with $b_i = a_{s+\ell+i}$) is legal.
\end{itemize}

If $\sum_{i=1}^{m} {a_i H_{m+1-i}}$ is a legal decomposition of $N$, we define the \textbf{number of summands} (of this decomposition of $N$) to be $a_1 + \cdots + a_m$.
\end{definition}

Informally, a legal decomposition is one where we cannot use the recurrence relation to replace a linear combination of summands with another summand and the coefficient of each summand is appropriately bounded.

\begin{remark}\label{rmk:kbonacciPLRS}
    The $k$-bonacci sequence $\{x_{n+1} \}_{n=1} ^\infty$ is a PLRS with $L=k$ and $c_1 = \dots = c_L = 1$.
\end{remark}

\subsection{Distribution of the Number of Summands}\label{subsection: number_of_summands}
Lemma \ref{Lemma: simple_bijection} relates the number of summands in vector Zeckendorf representations to those in $k$-Zeckendorf representations. Then, it is no surprise that the distribution of the number of summands converges to a Gaussian. 
\begin{proof}[Proof of Theorem \ref{thm: number_of_summands}] Follows from 
\cite[Theorems 2 and 3]{miller2014gaussian}.
\end{proof}
See \cite[Section 4]{millerwang2012CLT} for expressions characterizing the mean of the Gaussian for a general PLRS. For the case of vector Zeckendorf representations, we also propose an alternative way to find both the mean and the variance of the Gaussian by working with binary words.

\begin{definition} \label{def:w_nk}
Let $W_{n,k}$ be the set of binary words $w = w_1 \cdots w_n \in \{0,1\}^n$ such that
\begin{enumerate}
\item $w_1 = 1$ (the first bit is fixed to $1$);
\item $w$ contains no substring of $k$ consecutive $1$’s.
\end{enumerate}
\end{definition}

\begin{lemma} \label{lem:bijection}
For each $n \geq 1$ there exists a bijection
\[
\Phi_n:\ D_n \setminus D_{n-1} \ \longrightarrow\ W_{n,k}
\]
such that, for every $\ve{v} \in D_n \setminus D_{n-1}$ with the SR coefficient vector
\[
\varepsilon(\ve{v})~=~(\varepsilon_n,\varepsilon_{n-1},\dots,\varepsilon_1)\in\{0,1\}^n,
\]
\[
\varepsilon_j ~=~ \begin{cases}
1, &\text{if the $j^{\rm th}$ summand appears in the SR of $\ve{v}$},\\
0, &\text{otherwise},
\end{cases}
\]
we have $\Phi_n(\ve{v}) = w$, where $w_j = \varepsilon_j$ for $1 \leq j \leq n$.
The rules for the SR imply $\varepsilon_n = 1$ and no $k$ consecutive $1$’s appear in $(\varepsilon_n, \dots, \varepsilon_1)$, so $\Phi_n(\ve{v})$ is well-defined and is in $W_{n,k}$.
\end{lemma}

\begin{proof}
By definition of $D_n \setminus D_{n-1}$ and SR, each $\ve{v}$ has a unique coefficient vector $\varepsilon (\ve{v}) \in \{0,1\}^n$ indicating the presence or absence of each summand up to index $n$, with $\varepsilon_n = 1$.

The rules for the SR correspond to forbidding $k$ consecutive $1$’s among the coefficients. Mapping coefficients to bits produces a well-defined $w\in W_{n,k}$.

The map is injective by the uniqueness of the SR, and surjective by reversing the construction: read a word $w\in W_{n,k}$ as a valid coefficient vector, then form the SR.
\end{proof}

\begin{definition}\label{def:Xn}
For $n\ge 1$, we define the uniform probability measure on $W_{n,k}$:
\[
\Prob^{(W)}_n(A) \ = \ \frac{|A|}{|W_{n,k}|}, \qquad A \subseteq W_{n,k}.
\]
Define the random variable $X_n:W_{n,k}\to \mathbb{N}$ as
\[
X_n(w) \ = \ \text{the number of 1's in } w, \qquad w\in W_{n,k}.
\]
Using the bijection $\Phi_n$ from Lemma \ref{lem:bijection} and the uniform measure $\Prob_n$ on $D_n \setminus D_{n-1}$ from Definition \ref{defn: d_n}, the pushforward measure $(\Phi_n)_*\Prob_n$ equals $\Prob^{(W)}_n$. In particular, $X_n$ and $K_n$ have the same distribution:
\[
\Prob^{(W)}_n\bigl(X_n \ = \ t\bigr) \ = \ \Prob_n\bigl(K_n \ = \ t\bigr).
\]
\end{definition}

Counting binary words that avoid a pattern of $k$ consecutive $1$'s (also known as the substring $1^k$) is a classical problem in combinatorics (see, for example, \cite[Ch. I]{FlajoletSedgewick2009}); our case only differs slightly in that the first bit is fixed to be $1$.


We state a generating function for binary words that avoid $1^k$, in which $x$ marks the length and $y$ marks the number of $1$'s. We write
\[
F_k(x, y) \ = \ \sum_{n \ge 0} \ \sum_{m \ge 0} a_{n,m} \, x^n y^m,
\]
where $a_{n,m}$ counts binary words of length $n$ with exactly $m$ total $1$'s and with no substring $1^k$.

\begin{proposition}[From {\cite[Prop. 1]{BarilKirgizovVajnovszki2022}}]\label{prop:bkv}
For every fixed $k\ge 2$, the generating function for binary words avoiding $1^{k}$ is
\[
F_k(x,y) 
\ = \ \frac{ y \,\bigl( 1 \ - \ (x y)^{k} \bigr) }{\, y \ - \ x y^{2} \ - \ x y \ + \ (x y)^{k+1} \,}.
\]
Here $x$ marks each letter and $y$ marks each occurrence of the letter $1$.
\end{proposition}

We now pass from the setting of Proposition \ref{prop:bkv} to our model $W_{n,k}$ in Definition~\ref{def:w_nk}, where the first bit is fixed to $1$. Define
\[
F_k^{\mathrm{fix}}(x,y) 
\ = \ \sum_{n\ge 1}\ \sum_{m\ge 1} b_{n,m}\, x^{n} y^{m},
\]
where $b_{n,m}$ counts words in $W_{n,k}$ of length $n$ and with exactly $m$ ones. Thus $F_k^{\mathrm{fix}}(x,y)$ is the generating function for our fixed-first-bit model, while $F_k(x,y)$ is the generating for the unrestricted-first-bit model (which also includes the empty word).

\begin{proposition}\label{prop:fix}
For every fixed $k\ge 2$, the generating function for binary words that avoid $1^k$ and have $1$ fixed as the first bit is
\[
F_k^{\mathrm{fix}}(x,y) 
\ = \ \bigl( 1 \ - \ x \bigr)\, F_k(x,y) \ - \ 1 ,
\]
where $F_k(x,y)$ is the same as in Proposition~\ref{prop:bkv}. Equivalently,
\[
F_k^{\mathrm{fix}}(x,y) 
\ = \ \frac{ \bigl( 1 \ - \ x \bigr)\, y \,\bigl( 1 \ - \ (x y)^{k} \bigr) }{\, y \ - \ x y^{2} \ - \ x y \ + \ (x y)^{k+1} \,} \ - \ 1 .
\]
\end{proposition}

\begin{proof}
Every word counted by $F_k(x,y)$ either starts with $0$ or $1$, or is the empty word, which has weight $1$. Let $F_{k,0}(x,y)$ be the generating function for words starting with $0$, and $F_{k,0}(x,y)$ for words starting with $1$. Then $F_k(x,y)$ can be decomposed as
\[
F_k(x,y) 
\ = \ 1 \ + \ F_{k,0}(x,y) \ + \ F_{k,1}(x,y) .
\]

Adding a leading $0$ to any word that avoids $1^{k}$ results in a word that still avoids $1^{k}$ and contributes weight $x$ while contributing no factor of $y$. Similarly, removing the first bit from any word that starts with $0$ results in a word that avoids $1^{k}$. Therefore
\[
F_{k,0}(x,y) \ = \ x \, F_k(x,y).
\]

By definition, $F_k^{\mathrm{fix}}(x,y)$ is the generating function of words that avoid $1^{k}$ and start with $1$, so
\[
F_k^{\mathrm{start}=1}(x,y) \ = \ F_k^{\mathrm{fix}}(x,y).
\]
Hence
\[
F_k(x,y) 
\ = \ 1 \ + \ x\,F_k(x,y) \ + \ F_k^{\mathrm{fix}}(x,y),
\]
leading to
\[
F_k^{\mathrm{fix}}(x,y) 
\ = \ \bigl( 1 \ - \ x \bigr)\, F_k(x,y) \ - \ 1
\ = \ \frac{ \bigl( 1 \ - \ x \bigr)\, y \,\bigl( 1 \ - \ (x y)^{k} \bigr) }{\, y \ - \ x y^{2} \ - \ x y \ + \ (x y)^{k+1} \,} \ - \ 1 .
\]
\end{proof}

Having found the generating function, we may derive the exact mean and variance of $K_n$. Throughout this derivation, let $[x^n]G(x)$ denote the coefficient of $x^n$ in a series $G(x)$.

\begin{remark}
It is standard in analytic combinatorics (see, for example, \cite[Section III.2]{FlajoletSedgewick2009}) to identify combinatorial parameters with the random variables they induce under uniform sampling: a parameter $\chi$ counting a statistic in objects of length $n$ becomes a random variable $X_n$ when these objects are sampled uniformly. We cite relevant results involving parameters as if they were expressed in terms of random variables.
\end{remark}

The following result is a consequence of {\cite[Proposition III.2]{FlajoletSedgewick2009}}.
\begin{proposition}\label{prop:moment-identities}
Let $F(x,y) \ = \ \sum_{n\ge 0}\sum_{m\ge 0} f_{n,m}\,x^n y^m$ be a generating function where $f_{n,m}$ counts objects of size $n$ with parameter value $m$. For the random variable $Y_n$ obtained from uniform sampling of size-$n$ objects, with distribution
\[
\Prob(Y_n  =  m) \ = \ \frac{f_{n,m}}{\sum_{j\ge 0} f_{n,j}},
\]
The expectation and variance of $Y_n$ are: 
\[
\E[Y_n] 
\ = \ \frac{[x^n] \, \frac{\partial F}{\partial y}(x, y)\big|_{y=1}}{[x^n] \, F(x, 1)},
\qquad
\operatorname{Var}(Y_n)
\ = \ \frac{[x^n] \, \bigl(\frac{\partial^2 F}{\partial y^2} \ + \ \frac{\partial F}{\partial y}\bigr)(x, y)\big|_{y=1}}{[x^n] \, F(x, 1)}
\ - \
\bigl(\E[Y_n]\bigr)^2.
\]
\end{proposition}

From here we apply Proposition \ref{prop:moment-identities} for $Y_n = X_m$ and $F \ = \ F_k^{\mathrm{fix}}$ from Proposition \ref{prop:fix}.

\begin{theorem}[Exact mean and variance]
For every $k \ge 2$ and $n \ge 1$, let
\[
A_k(x) \ = \ F_k^{\mathrm{fix}}(x,1),\] 
\[B_k(x) \ = \ \partial_y F_k^{\mathrm{fix}}(x,y)\big|_{y=1}, and 
\]
\[C_k(x) \ = \ \bigl(\partial_y^2 F_k^{\mathrm{fix}} \ + \ \partial_y F_k^{\mathrm{fix}}\bigr)(x,y)\big|_{y=1}.\]

Then the mean and variance of $X_n$ under the uniform distribution on $W_{n,k}$ can be calculated using the expressions we have just defined as follows: 
\[
\E[X_n] \ = \ \frac{ [x^n]\,B_k(x) }{ [x^n]\,A_k(x) },
\qquad
\Var(X_n) \ = \ \frac{ [x^n]\,C_k(x) }{ [x^n]\,A_k(x) } \ - \ \left( \frac{ [x^n]\,B_k(x) }{ [x^n]\,A_k(x) } \right)^{\!2}.
\]

The explicit forms of $A_k, B_k, C_k$ are listed below:
\[
A_k(x) 
\ = \ \frac{ x \ - \ x^{k} }{ (1 \ - \ x)\,\Delta_k(x) }
\ = \ \frac{ x \ - \ x^{k} }{ 1 \ - \ 2x \ + \ x^{k+1} },
\]

\[
\begin{aligned}
B_k(x) 
\ = &\ \frac{1}{\bigl((1 \ - \ x)\,\Delta_k(x)\bigr)^{2}}
\Bigl[
\ \bigl(x \ - \ k x^{k}\bigr)\,(1 \ - \ x)\,\Delta_k(x) \\
&\ - \ \bigl(x \ - \ x^{k}\bigr)\,\Bigl( -\,x\,\Delta_k(x) \ - \ (1 \ - \ x)\,\sum_{r=1}^{k-1} r\,x^{r+1} \Bigr)
\Bigr],
\end{aligned}
\]

\[
\begin{aligned}
C_k(x) \ = & \ \frac{1}{\bigl((1 \ - \ x)\,\Delta_k(x)\bigr)^{3}}
\Bigl[ \bigl(-\,k(k-1)\,x^{k}\bigr)\,\bigl((1 \ - \ x)\,\Delta_k(x)\bigr)^{2} \\
&\ - \ \bigl(x \ - \ x^{k}\bigr)\,\Bigl( (1 \ - \ x)\,\bigl(-\,\sum_{r=2}^{k-1} r(r-1)\,x^{r+1}\bigr) \ + \ 2x\,\sum_{r=1}^{k-1} r\,x^{r+1} \Bigr)\,(1 \ - \ x)\,\Delta_k(x) \\
&\ - \ 2\,\bigl(x \ - \ k x^{k}\bigr)\,(1 \ - \ x)\,\Delta_k(x)\,\Bigl( -\,x\,\Delta_k(x) \ - \ (1 \ - \ x)\,\sum_{r=1}^{k-1} r\,x^{r+1} \Bigr) \\
&\ + \ 2\,\bigl(x \ - \ x^{k}\bigr)\,\Bigl( -\,x\,\Delta_k(x) \ - \ (1 \ - \ x)\,\sum_{r=1}^{k-1} r\,x^{r+1} \Bigr)^{2}
\Bigr] \ + \ B_k(x),
\end{aligned}
\]

where
\[
\Delta_k(x) \ = \ 1 \ - \ \sum_{j=1}^{k} x^{j}.
\]
\end{theorem}

\begin{proof}
By Proposition \ref{prop:fix},
\[
F_k^{\mathrm{fix}}(x, y) \ = \ \frac{x y \, \bigl(1 \ - \ (x y)^{k-1} \bigr)}{\bigl(1 \ - \ x y\bigr) \, \Delta_k(x, y)},
\qquad
\Delta_k(x, y) \ = \ 1 \ - \ x \ - \ x^2 y \ - \ \cdots \ - \ x^k y^{k-1}.
\]
To make this simpler, set
\[
N(x, y) \ = \ x y \ - \ x^k y^k, 
\qquad
V(x, y) \ = \ \bigl(1 \ - \ x y\bigr) \, \Delta_k(x, y).
\]
From here we can write $F_k$ as $F_k^{\mathrm{fix}} \ = \ N/V$. 

Now we evaluate the derivatives at $y \ = \ 1$:
\[
N(x, 1) \ = \ x \ - \ x^k, 
\qquad
\frac{\partial N}{\partial y}(x, y) \ = \ x \ - \ k x^k y^{k-1} \ \Rightarrow \ \frac{\partial N}{\partial y}(x, 1) \ = \ x \ - \ k x^k,
\]
\[
\frac{\partial^2 N}{\partial y^2}(x, y) \ = \ -k(k - 1) x^k y^{k-2} \ \Rightarrow \ \frac{\partial^2 N}{\partial y^2}(x, 1) \ = \ -k(k - 1) x^k.
\]
For the denominator,
\[
V(x, 1) \ = \ (1 \ - \ x) \, \Delta_k(x),
\]
\[
\frac{\partial V}{\partial y}(x, y) \ = \ -x \, \Delta_k(x, y) \ + \ (1 \ - \ x y) \, \frac{\partial \Delta_k}{\partial y}(x, y),
\]
so that
\[
\frac{\partial \Delta_k}{\partial y}(x, y) \ = \ -x^2 \ - \ 2 x^3 y \ - \ \cdots \ - \ (k - 1) x^k y^{k-2},
\]
\[
\frac{\partial V}{\partial y}(x, 1) \ = \ -x \, \Delta_k(x) \ - \ (1 \ - \ x) \sum_{r = 1}^{k - 1} r \, x^{r + 1}.
\]
Differentiating again,
\[
\frac{\partial^2 V}{\partial y^2}(x, y) \ = \ (1 \ - \ x y) \, \frac{\partial^2 \Delta_k}{\partial y^2}(x, y) \ - \ 2 x \, \frac{\partial \Delta_k}{\partial y}(x, y),
\]
with
\[
\frac{\partial^2 \Delta_k}{\partial y^2}(x, y) \ = \ -2 x^3 \ - \ 3 \cdot 2 \, x^4 y \ - \ \cdots \ - \ (k - 1)(k - 2) x^k y^{k-3},
\]
so that
\[
\frac{\partial^2 V}{\partial y^2}(x, 1) 
\ = \ (1 \ - \ x) \Bigl(-\sum_{r = 2}^{k - 1} r(r - 1) x^{r + 1}\Bigr) \ + \ 2 x \sum_{r = 1}^{k - 1} r \, x^{r + 1}.
\]

Now we compute $A_k$, $B_k$, $C_k$. By definition,
\[
A_k(x) \ = \ \frac{N(x, 1)}{V(x, 1)} \ = \ \frac{x \ - \ x^k}{(1 \ - \ x) \, \Delta_k(x)},
\]
which also equals $(x \ - \ x^k)/(1 \ - \ 2x \ + \ x^{k+1})$ since $(1 \ - \ x) \, \Delta_k(x) \ = \ 1 \ - \ 2x \ + \ x^{k+1}$.

For the first derivative, the quotient rule gives
\[
\frac{\partial}{\partial y}\!\left(\frac{N}{V}\right) \ = \ \frac{\frac{\partial N}{\partial y} \cdot V \ - \ N \cdot \frac{\partial V}{\partial y}}{V^2}.
\]
Evaluating this at $y \ = \ 1$ and substituting in the four expressions $N(x, 1)$, $\frac{\partial N}{\partial y}(x, 1)$, $V(x, 1)$, $\frac{\partial V}{\partial y}(x, 1)$ results in 
\[
\begin{aligned}
B_k(x) \ = &\ \frac{1}{\bigl((1 \ - \ x) \, \Delta_k(x)\bigr)^2}
\Bigl[
\ \bigl(x \ - \ k x^k\bigr) \, (1 \ - \ x) \, \Delta_k(x) \\
&\ - \ \bigl(x \ - \ x^k\bigr) \, \Bigl(-x \, \Delta_k(x) \ - \ (1 \ - \ x) \sum_{r = 1}^{k - 1} r \, x^{r + 1} \Bigr)
\Bigr].
\end{aligned}
\]

For the second derivative, we differentiate again:
\[
\frac{\partial^2}{\partial y^2}\!\left(\frac{N}{V}\right)
\ = \ \frac{\frac{\partial^2 N}{\partial y^2} \cdot V^2 \ - \ N \cdot \frac{\partial^2 V}{\partial y^2} \cdot V \ - \ 2 \frac{\partial N}{\partial y} \cdot V \cdot \frac{\partial V}{\partial y} \ + \ 2 N \cdot \left(\frac{\partial V}{\partial y}\right)^2}{V^3}.
\]
Evaluating at $y \ = \ 1$ and substituting $\frac{\partial^2 N}{\partial y^2}(x, 1)$, $\frac{\partial^2 V}{\partial y^2}(x, 1)$ together with the derived expression gives:
\[
\begin{aligned}
\frac{\partial^2 F_k^{\mathrm{fix}}}{\partial y^2}(x, 1) \ = & \  \frac{1}{\bigl((1 \ - \ x) \, \Delta_k(x)\bigr)^3}
\Bigl[ \bigl(-k(k - 1) x^k\bigr) \, \bigl((1 \ - \ x) \, \Delta_k(x)\bigr)^2 \\
&\ - \ \bigl(x \ - \ x^k\bigr) \, \Bigl( (1 \ - \ x) \, \bigl(-\sum_{r = 2}^{k - 1} r(r - 1) x^{r + 1}\bigr) \ + \ 2 x \sum_{r = 1}^{k - 1} r \, x^{r + 1} \Bigr) \, (1 \ - \ x) \, \Delta_k(x) \\
&\ - \ 2 \bigl(x \ - \ k x^k\bigr) \, (1 \ - \ x) \, \Delta_k(x) \, \Bigl( -x \, \Delta_k(x) \ - \ (1 \ - \ x) \sum_{r = 1}^{k - 1} r \, x^{r + 1} \Bigr) \\
&\ + \ 2 \bigl(x \ - \ x^k\bigr) \, \Bigl( -x \, \Delta_k(x) \ - \ (1 \ - \ x) \sum_{r = 1}^{k - 1} r \, x^{r + 1} \Bigr)^2
\Bigr].
\end{aligned}
\]
By definition,
\[
C_k(x) \ = \ \frac{\partial^2 F_k^{\mathrm{fix}}}{\partial y^2}(x, 1) \ + \ B_k(x).
\]

Applying Proposition \ref{prop:moment-identities} to $F \ = \ F_k^{\mathrm{fix}}$ gives:
\[
\E[X_n] \ = \ \frac{[x^n] \, B_k(x)}{[x^n] \, A_k(x)},
\qquad
\operatorname{Var}(X_n) \ = \ \frac{[x^n] \, C_k(x)}{[x^n] \, A_k(x)} \ - \ \left(\frac{[x^n] \, B_k(x)}{[x^n] \, A_k(x)}\right)^{\!2},
\]
as was to be shown.
\end{proof}

\subsection{Distribution of Gaps Between Summands}
Again, Lemma \ref{Lemma: simple_bijection} allows us to reduce our analysis to the one-dimensional case. There are results on limiting gap probabilities, longest gaps, and much more (see \cite{beckwith2013average, bower2015gaps}).

\begin{proof}[The Proof of Theorem \ref{thm:avegapsbulk}] follows from 
\cite[Theorem 1.5]{beckwith2013average} and \cite[Theorem 1.5]{bower2015gaps}.
\end{proof}

\subsection{Summand Minimality}

We utilize the following definitions and result from \cite{cordwell2018summand}.

\begin{definition}
    For a PLRS given by the recurrence relation $H_{n+1} \ = \ c_1 H_n + \cdots + c_L H_{n+1-L},$ we call $\sigma=(c_1,c_2,..,c_L)$ the \textbf{\textit{signature}} of the PLRS.
\end{definition}
As we observed in Remark \ref{rmk:kbonacciPLRS}, $\{x_{n+1} \}_{n=1} ^\infty$ is a PLRS with $L=k$ and $c_1 = \dots = c_L = 1$. Kologlu et al. \cite[Theorem 1.3]{kologlu2011summands} proved that every positive integer $n$ has a unique legal decomposition associated to a given PLRS, called the \textbf{\textit{generalized Zeckendorf decomposition}} of $n$. In the case of the $k$-bonacci sequence, this is the unique decomposition given by Theorem \ref{thm:kbonacci-zeckendorf}. We will say that a PLRS $\{H_n\}_{n=1}^\infty$ is \textbf{\textit{summand minimal}} if no representation of any positive integer $n$ as a linear combination with nonnegative coefficients of terms in $\{H_n\}_{n=1}^\infty$ uses fewer summands than the generalized Zeckendorf decomposition of $n$.

\begin{theorem}\cite[Theorem 1.1]{cordwell2018summand}\label{thm:plrs-summand-minimality}
    A PLRS with signature $\sigma=(c_1,c_2, \dots , c_L)$ is summand minimal if and only if $c_1 \geq c_2 \geq \cdots \geq c_L$.
\end{theorem}

Theorem \ref{thm:plrs-summand-minimality} gives the following as an immediate corollary.

\begin{corollary}\label{lem:kbonacci-number-minimality}
    The $k$-bonacci number decomposition of $n$ given by Theorem \ref{thm:kbonacci-zeckendorf} is summand minimal.
\end{corollary}

We are now able to prove the following.

\begin{proof}[Proof of Theorem \ref{thm: vec_sum_minimality}]
    Consider any other representation $\ve{v}=\sum_{i \geq 1}c'_i\ve{X}_{-i}$ of $\ve{v}$ as a linear combination of $k$-bonacci vectors with nonnegative integer coefficients. Suppose the SR of $\ve{v}$ contains $d$ summands and assume, for the sake of contradiction, that this representation $\ve{v}=\sum_{i \geq 1}c'_i\ve{X}_{-i}$ contains $\sum_{i \geq 1}c_i'=c<d$ summands. Let $m$ be the maximal integer such that $\ve{X}_{-m}$ is present in one of the representations.
    By Lemma \ref{lem:projection}, we have $$S_{m+1}(\sum_{i=1}^pc_i\ve{X}_{-i})~=~\sum_{i=1}^pc_ix_{m+1-i} \pmod{x_{m+1}}$$ and $$S_{m+1}(\sum_{i=1}^pc'_i\ve{X}_{-i})~=~\sum_{i=1}^pc'_ix_{m+1-i} \pmod{x_{m+1}}.$$ Observe that $\sum_{i=1}^pc_ix_{m+1-i} \pmod{x_{m+1}}$ and $\sum_{i=1}^pc'_ix_{m+1-i} \pmod{x_{m+1}}$ contain, respectively, $d$ and $c$ summands. Now, by definition of the SR, $\sum_{i=1}^pc_ix_{m+1-i} \pmod{x_{m+1}}$ has no $k$ consecutive terms and $c_i \in \{0,1\}$ for every $i$, so $\sum_{i=1}^pc_ix_{m+1-i} \pmod{x_{m+1}}$ must be the $k$-bonacci number decomposition of $S_{m+1}(\ve{v})$. Then Lemma \ref{lem:kbonacci-number-minimality} implies $d \leq c$.
\end{proof}

\section{Future Work}

A natural direction for further research concerns the behavior of the vector representation algorithm for general $k$. We proved that for $k=3$, the number of steps required is logarithmic in the size of the vector, but for larger $k$, this remains conjectural (see Conjecture~\ref{conj:large-steps-j-bound}). Formally, it is conjectured that for all $k \ge 3$, there exist constants $c_k, d_k>0$ such that the algorithm completes in at most $c_k \log |\mathbf{v}| + d_k$ steps. Proving this conjecture for higher $k$ and understanding the underlying combinatorial mechanisms presents an important avenue for future research.

During our work, we also observed that the sets ${D}_n$ increasingly resemble Rauzy fractals as $n$ grows, displaying intricate self-similar structures. Despite this clear numerical and visual evidence, we were not able to provide a rigorous proof of their convergence. Establishing such a result, as well as understanding the precise geometric and combinatorial mechanisms behind this behavior, remains an interesting and challenging direction for future research.

\appendix

\section{Algorithm Complexity}\label{appendix:complexity}

In order to compare to the efficiency of our algorithms, we first provide the following complexity analysis of the recursive algorithm given in \cite{anderson2011multidimensional}.

\begin{proof}[Proof of Lemma \ref{lem:old-algo-complexity}]
The algorithm described in \cite{anderson2011multidimensional} follows a recursive approach to compute the Satisfying Representation (SR) of a vector $\ve{v}$. We distinguish two cases based on the signs of the coordinates of $\ve{v}$. \\

\textit{Recursive Algorithm:}
\begin{itemize}
    \item \textbf{Case 1:} If $\ve{v}$ has any positive coordinate, say $v_i > 0$, define $\ve{w} = \ve{v} - \ve{e}_i$. Then by definition, 
    \[
        SR(\ve{v}) ~=~ SR(\ve{w}) + \ve{e}_i.
    \]
    Since $\ve{e}_i$ is a standard basis vector, this expression is either:
    \begin{itemize}
        \item a Nearly Satisfying Representation (NSR),
        \item a valid SR, or
        \item an “almost SR,” where the only violation is a single block of 1’s of length at most $2k - 1$.
    \end{itemize}
    \smallskip
    \item \textbf{Case 2:} If all coordinates of $\ve{v}$ are non-positive, define
    \[
        \ve{w} ~=~ \ve{v} + \ve{X}_{-k} ~=~ \ve{v} - (\ve{e}_1 + \ve{e}_2 + \dots + \ve{e}_{k-1}).
    \]
    Then we have:
    \[
        SR(\ve{v}) ~=~ SR(\ve{w}) + \ve{X}_{-k}.
    \]
    The vector $\ve{w}$ has strictly larger entries than $\ve{v}$, so we recurse on $\ve{w}$. The resulting expression $SR(\ve{w}) + \ve{X}_{-k}$ is either a valid SR, or one of the following:
    \begin{itemize}
        \item an NSR, or
        \item an “almost SR” with a single block of 1’s of length at most $2k - 1$.
    \end{itemize}
    In either of these two cases, we proceed with a normalization step using the “borrow-carry” operation as described in Lemma 4.
\end{itemize}
\smallskip
When $\ve{v}$ has positive coordinates, each operation $\ve{v} \to \ve{v} - \ve{e}_i$ reduces the $\ell_1$ norm of $\ve{v}$ by 1. Therefore, in the worst case, we perform this operation at most $O(\|\ve{v}\|_1)$ times. Appending $\ve{e}_i$ to the representation costs $O(1)$ per step. When we apply Lemma 4 to reduce blocks of 1’s (of length at most $2k - 1$), we know from the Lemma that a constant number (at most two) of “borrow-carry” operations suffices to restore the SR condition. Each such operation affects a block of length $O(k)$, so each normalization step takes $O(k)$ time. If the resulting expression is an NSR of length $L$ (i.e., involving $L$ summands), and we need to normalize it into an SR, then the total cost of normalization is $O(L \cdot k)$.

Assuming we begin with a vector $\ve{v} \in \mathbb{Z}^{k-1}$:
\begin{itemize}
    \item We perform at most $O(\|\ve{v}\|_1)$ recursive steps.
    \item Converting each NSR to an SR costs $O(L \cdot k)$ time.
\end{itemize}

Therefore, the total worst-case time complexity is:
\[
    O(k \cdot L_{v_1} \cdot \|\ve{v}\|_1).
\]
\end{proof}

For comparison, we calculate the complexity of the small steps algorithm and the large steps algorithm.

\begin{proof}[Proof of Proposition \ref{prop:complexity-small-steps}]
    We list the steps of the algorithm.
    \begin{enumerate}[(i)]
        \item Calculate $j=j_{ssb}$.
        \item Calculate the first $j$ $k$-bonacci numbers.
        \item Apply the map $S_{j+1}$.
        \item Apply the greedy algorithm to $S_{j+1}(\ve{v})$.
    \end{enumerate}
    Now, we calculate the complexity of each step.
    \begin{enumerate}[(i)]
        \item Finding $j=j_{ssb}$ takes $O(k)$ operations.
        \item Calculating one $k$-bonacci number takes $O(k)$ operations, so calculating $j$ $k$-bonacci vectors takes $O(kj)$ operations.
        \item Evaluating $S_{j+1}$ requires a dot product modulo $x_{j+1}$. The dot product takes $O(k)$ operations, and since $x_{j+1}$ is approximately $\norm{\ve{v}}_\infty$ times smaller than the dot product, using repeated subtractions gives us that the modulo calculation takes $O(\norm{\ve{v}}_\infty)$ operations.
        \item Applying the greedy algorithm requires $O(j)$ checks and $O(j)$ computations, so it takes $O(j)$ operations.
    \end{enumerate}
    Since $j=O(k^2 \norm{\ve{v}}_\infty)$, the whole algorithm has time complexity $O(k^3 \norm{\ve{v}}_\infty)$.
\end{proof}

\begin{proof}[Proof of Proposition \ref{prop:complexity-large-steps}]
    As before, we list the steps of the algorithm.
    \begin{enumerate}[(i)]
        \item Calculate all of the $k$-bonacci vectors $\ve{X}_{-n}$ with $\norm{\ve{X}_{-n}}_2 \leq 2\norm{\ve{v}}_2$. This takes $O(\log \norm{\ve{v}}_\infty)$ steps.
        \item Inductively:
        \begin{enumerate}[(a)]
            \item Find the closest $k$-bonacci vector to $\ve{v}_n$. This takes $O(\log \norm{\ve{v}}_2)$ steps.
            \item Set $\ve{v}_{n+1}$. This takes $O(1)$ steps.
            \item Check if $\norm{\ve{v}_{n+1}}_2 < \norm{\ve{v}_n}_2$. This also takes $O(1)$ steps.
        \end{enumerate}
        \item Apply the greedy algorithm. We know that this takes $O(j+\norm{\ve{v}}_\infty)$ operations.
    \end{enumerate}
    
    Since $j=O(\log \norm{\ve{v}}_\infty)$, the greedy algorithm takes $O(\norm{\ve{v}}_\infty)$ operations. The large step algorithm for finding $j_{lsb}$ is $O((\log \norm{\ve{v}}_\infty)^2)$, so the whole algorithm for each $k$ is linear in $\norm{\ve{v}}_\infty$.
\end{proof}

\subsection*{Disclosure statement}

No conflict of interest has been reported by the author(s).


\medskip
\noindent MSC2020: 11B39


\end{document}